\documentclass[12pt]{article} 
\usepackage{amsmath,amssymb,latexsym,amsthm, enumerate,amscd,hyperref} %\usepackage[abbrev]{amsrefs} 
\usepackage{graphicx}
\usepackage{pinlabel}
\usepackage{todonotes}
\usepackage{color}
\newtheoremstyle{dotless}{}{}{\itshape}{}{\bfseries}{}{ }{}
\usetikzlibrary{decorations.markings}

\usepackage{overpic}

\newtheorem{Theorem}{Theorem}[section]

\newtheorem{Question}[Theorem]{Question} 
\newtheorem{lemma}[Theorem]{Lemma}

\newtheorem*{theorem*}{Theorem}

\theoremstyle{definition} 

\newtheorem{remark}[Theorem]{Remark}

\newtheorem*{Remark}{Remark} 
 
\newtheorem*{Notation}{Notation}

\DeclareMathOperator{\piprod}{\raisebox{-0.1em}{\huge{$\pi$}}\kern -0.2em}

\newcommand{\rr}{{\mathbb R}}

\newcommand{\zz}{{\mathbb Z}}

\newcommand{\wtX}{\widetilde{X}}

\newcommand{\Isom}{\operatorname{Isom}}

\def\clap#1{\hbox to 0pt{\hss#1\hss}}

\newcommand{\comment}[1]{} 
 
\newcommand{\gG}{\Gamma} 
 
\newcommand{\gO}{\Omega}

\newcommand{\Itin}{Itin_p}

\newcommand{\SB}{{SB}}

	{ 
\end{enumerate}
} 

	{ 
\end{enumerate}
} 

\newenvironment{enumeratei'}{ 
\begin{enumerate}[\upshape (i)$'$]}
	{ 
\end{enumerate}
} 

\newenvironment{enumerate1'}{ 
\begin{enumerate}[\upshape (1)$'$]}
	{ 
\end{enumerate}
}

\newenvironment{enumeratea'}{ 
\begin{enumerate}[\upshape (a)$'$]}{ 
\end{enumerate}
}

\usepackage{color}
\usepackage[outerbars,color]{changebar}
\usepackage{ifthen}

\newboolean{usecolor}
\setboolean{usecolor}{true}
\ifthenelse{\boolean{usecolor}}
{
  \definecolor{colore}{cmyk}{0,1,0.6,0}
  \definecolor{coloregen}{cmyk}{0.7,0,1,0}
  \definecolor{coloresimo}{cmyk}{1,0.6,0,0}
}
{
  \definecolor{colore}{cmyk}{0,0,0,1}
  \definecolor{coloregen}{cmyk}{0,0,0,1}
  \definecolor{coloresimo}{cmyk}{0,0,0,1}
}

\numberwithin{equation}{section} 
\newcommand{\CAT}{\operatorname{CAT}}
\definecolor{amethyst}{rgb}{0.6, 0.4, 0.8}

\newcommand{\hide}[1]{}
\newcommand{\p}{\partial_{\infty}}

\begin{document}

\title{Nonplanar graphs in boundaries of CAT(0) groups}

\author{Kevin Schreve
\and {Emily Stark}  
}

\date{\today} \maketitle

\begin{abstract} 
%\smallskip

\noindent
Croke and Kleiner constructed two homeomorphic locally $\CAT(0)$ complexes whose universal covers have visual boundaries that are not homeomorphic.  We construct two homeomorphic locally $\CAT(0)$ complexes so that the visual boundary of one universal cover contains a nonplanar graph, while the visual boundary of the other does not. In contrast, we prove for any two locally CAT(0) metrics on the Croke-Kleiner complex, if a finite graph embeds in the visual boundary of one universal cover, then the graph embeds in the visual  boundary of the other. 

\medskip 

\noindent
\textbf{AMS classification numbers}. Primary: 20F65. 
Secondary: 57M60, 20F67, 20E06.
%% 20F65 - Geometric group theory
%% 57M60 - Group actions in low dimensions
%% 20F67 - Hyperbolic groups and nonpositively curved groups
%% 20E06 - Free products, free products with amalgamation, HNN extensions, and generalizations
%%%% 20F36 - Braid groups, Artin groups
%%%% 20F55 - Reflection and Coxeter groups
%%%% 57S30 - Discontinuous groups of transformations
%%%% 57Q35 - Embeddings and immersions
%%%% 20J06 - Cohomology of groups
%%%% 32S22 - Relations with arrangments of hyperplanes 

\smallskip

 \end{abstract}

\section{Introduction}

If $X$ is a Gromov hyperbolic space, there is a naturally defined boundary at infinity $\partial_\infty X$, and any quasi-isometry of $X$ induces a self-homeomorphism of $\partial_\infty X$. Moreover, if $G$ is a word-hyperbolic group, then any two boundaries of $G$ are $G$-equivariantly homeomorphic. Croke and Kleiner~\cite{crokekleiner} showed the same phenomena does not occur for $\CAT(0)$ groups; i.e. there is a group which acts geometrically (properly and cocompactly by isometries) on two $\CAT(0)$ complexes with non-homeomorphic visual boundaries. Later, Wilson~\cite{wilson} showed that the Croke-Kleiner examples admit uncountably many non-homeomorphic boundaries; see also \cite{bowersruane,crokekleiner02,qing}. 

One can still ask what properties of the visual boundary are well-defined invariants of a $\CAT(0)$ group. For example, the topological dimension of a $\CAT(0)$-group is a quasi-isometry invariant \cite{bestvinamess,geogheganontaneda}. %It also follows from Papasoglu~\cite{papasoglu} and Papasoglu--Swenson~\cite{papasogluswenson} that for any one-ended CAT(0) group the existence of cut pairs in the boundary is a quasi-isometry invariant. 
%\nota{I still think this papasoglu swenson thing is true but I don't think we need it}
In a different direction, Guilbault and Mooney~\cite{guilbaultmooney} proved that all boundaries of the Croke-Kleiner examples are $G$-equivariantly cell-like equivalent. 

In this paper, we show that the existence of a nonplanar graph in a $\CAT(0)$ boundary is not a well-defined invariant for a $\CAT(0)$ group. 

\begin{Theorem}\label{t:main}
There exist two homeomorphic locally $\CAT(0)$ complexes $X$ and $X'$ so that the visual boundary of $\widetilde X$ contains a nonplanar graph and the visual boundary of $\widetilde X'$ does not.
\end{Theorem}

On the other hand, we prove that the homeomorphism type of the boundaries for the Croke-Kleiner examples cannot be detected by finite graphs. These complexes, denoted $X_{CK}$, are constructed by gluing two flat tori $T_1$ and $T_2$ onto a third flat torus $T_0$ along simple closed curves which generate $\pi_1(T_0)$.

\begin{Theorem}\label{theorem_main2}
Suppose $X_1$ and $X_2$ are locally CAT(0) complexes homeomorphic to $X_{CK}$. If $\gG$ is a finite graph contained in $\partial_\infty \widetilde{X}_1$, then there is an embedding of the graph $\gG$ into $\partial_\infty \widetilde{X}_2$. 
\end{Theorem}

Our interest in (non)-planarity of the visual boundary is motivated by the following two questions, neither of which we can answer. 

\begin{Question}\label{q:1}
Suppose a group $G$ acts geometrically on a CAT(0) space $X$ so that $\partial_\infty X$ is planar. Does $G$ have a finite-index subgroup which is a $3$-manifold group?
\end{Question}

\begin{Question}\label{q:2}
Can a group $G$ act geometrically on two CAT(0) spaces $X$ and $X'$ so that $\partial_\infty X$ is planar and $\partial_\infty X'$ is nonplanar?
\end{Question}

Question~\ref{q:1} was asked by Ha\"{i}ssinsky for hyperbolic groups. A positive answer implies the Cannon Conjecture~\cite[Conjecture 11.34]{cannon}. Ha\"{i}ssinsky proved that the answer to Question \ref{q:1} is positive if $G$ is hyperbolic and cubulated ~\cite{haissinsky}. It is also necessary to ask for a finite-index subgroup, as there are torsion-free hyperbolic and $\CAT(0)$ groups with planar boundary which are not $3$-manifold groups but have $3$-manifold groups as finite-index subgroups~\cite{kapovichkleiner,hruskastarktran}. 

Regarding Question \ref{q:2}, Swenson~\cite{swenson99} showed that the boundary of a one-ended CAT(0) group has no global cut points provided the group does not contain an infinite torsion subgroup. In this case, if the boundary is also locally connected, the existence of an embedded nonplanar graph in the boundary is equivalent to nonplanarity of the boundary by a theorem of Claytor~\cite{claytor}. Our groups are torsion-free, but the boundaries in Theorem~\ref{t:main} that we construct are not locally connected, and we prove in Theorem~\ref{theorem_nonplanar} the boundaries that do not contain nonplanar graphs are also nonplanar.  

One relation between planarity of the boundary and $3$-manifold groups comes from a paper of Bestvina, Kapovich, and Kleiner \cite{bestvinakapovichkleiner}. A very special case of their main theorem implies that if $G$ acts geometrically on a CAT(0) space $X$ and $\partial_\infty X$ contains a nonplanar graph, then $G$ is not even quasi-isometric to a $3$-manifold group. It follows that our examples are not quasi-isometric to $3$-manifold groups.

This paper is organized as follows. In Section 2 we introduce and illustrate our main example; see Figure~\ref{figure:main_example}. Our example is formed by gluing tori to the torus boundary components of a $3$-manifold that is the product of a surface with boundary and a circle. %Onto one of the torus boundary components of the $3$-manifold we attach two tori along simple closed curves that span the boundary torus. 
%We show there is a metric on the $3$-manifold so that geodesic representatives of the gluing curves on the boundary torus meet at any angle. If these curves meet at right-angles, we say the complex is {\it right-angled}, and otherwise it is not. 
In Section 3 we determine the local path components of the visual boundary; we prove these depend on the locally $\CAT(0)$ metric on the complex. In Section 4 we show for any locally $\CAT(0)$ metric on the complex, the visual boundary is non-planar. In Section 5 we show the existence of nonplanar graphs depends on the locally CAT(0) metric. In Section 6 we analyze the finite subgraphs of the Croke-Kleiner boundaries.

\subsection*{Acknowledgements} The authors are thankful for helpful conversations with Ric Ancel, Radhika Gupta, Chris Hruska, and Genevieve Walsh. The authors were supported by a Faculty Allies for Diversity in Graduate Education grant, which helped finance the second author's trip to the University of Michigan. The second author was partially supported at the Technion by a Zuckermann STEM Leadership Postdoctoral Fellowship. This material is based upon work supported by the National Science Foundation under Award No. 1704364.

\section{Main example}

\begin{figure}
\begin{center}
 	 \begin{overpic}[scale=.55,tics=5]{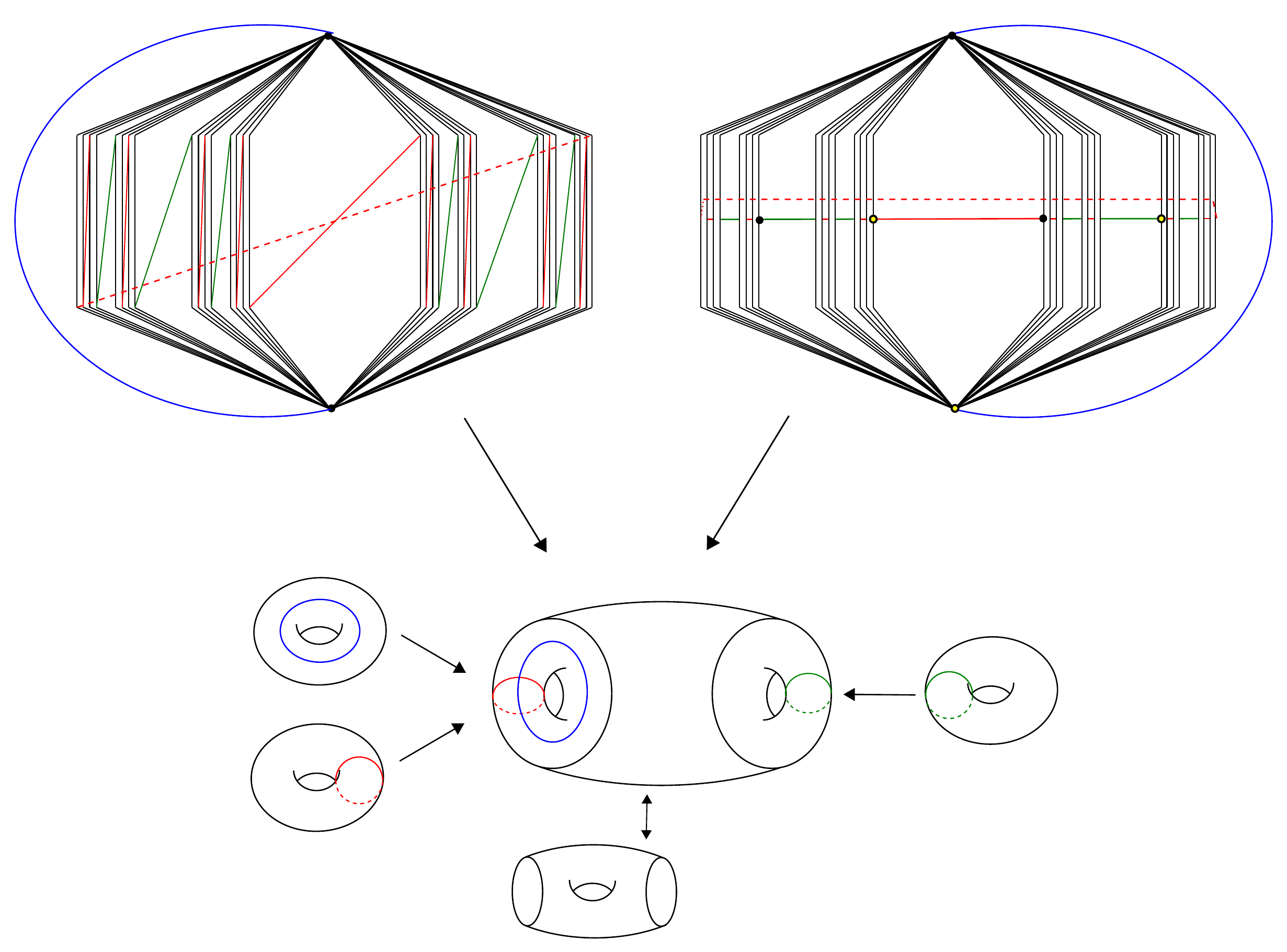} 
 	 \put(54,3.5){$\times$ \, $S^1$}
 	 \put(51.3,9.5){\small{$\cong$}}
 	 \put(5,20){$X$}
 	 \put(15,12.5){$T_0$}
 	 \put(83.5,20){$T_1$}
 	 \put(15,24){$T_2$}
 	 \put(47,24){\small{$\Sigma \times S^1$}}
 	 \put(42.5,5.7){\footnotesize{$\Sigma$}}
 	 \put(16,35){\small{not right-angled}}
 	 \put(62,35){\small{right-angled}}
 	 \put(73,72){\small{$a$}}
 	 \put(59.5,54.5){\small{$b$}}
 	 \put(79,54.5){\small{$c$}}
 	 \put(73,40){\small{$x$}}
 	 \put(68.5,54.5){\small{$y$}}
 	 \put(88,54.5){\small{$z$}}
	 \end{overpic}
	 \caption{{\small The main example $X$ and a subset of the boundary of $\widetilde{X}$. The space $X$ is {\it right-angled} if the tori $T_0$ and $T_2$ are glued to geodesics in a boundary torus of  $\Sigma \times S^1$ that meet at a right angle. 
	 On the right $X$ is right-angled, and on the left $X$ is not. The suspension of the Cantor set in black is the boundary  of a sub-block which covers $\Sigma \times S^1$. The red and green paths are subsets of the visual boundary $\widetilde T_i$. On the right these paths yield an embedded circle in $\p X$, and on the left they do not. Using this embedded circle, it is simple to find a $K_{3,3}$ subgraph of the right-angled boundary; the vertex sets are $\{a,b,c\}$ and $\{x,y,z\}$. Most of the work in this paper is to show that there are no $K_{3,3}$ subgraphs of the non right-angled boundary. } }
	 \label{figure:main_example}
\end{center}
\end{figure}
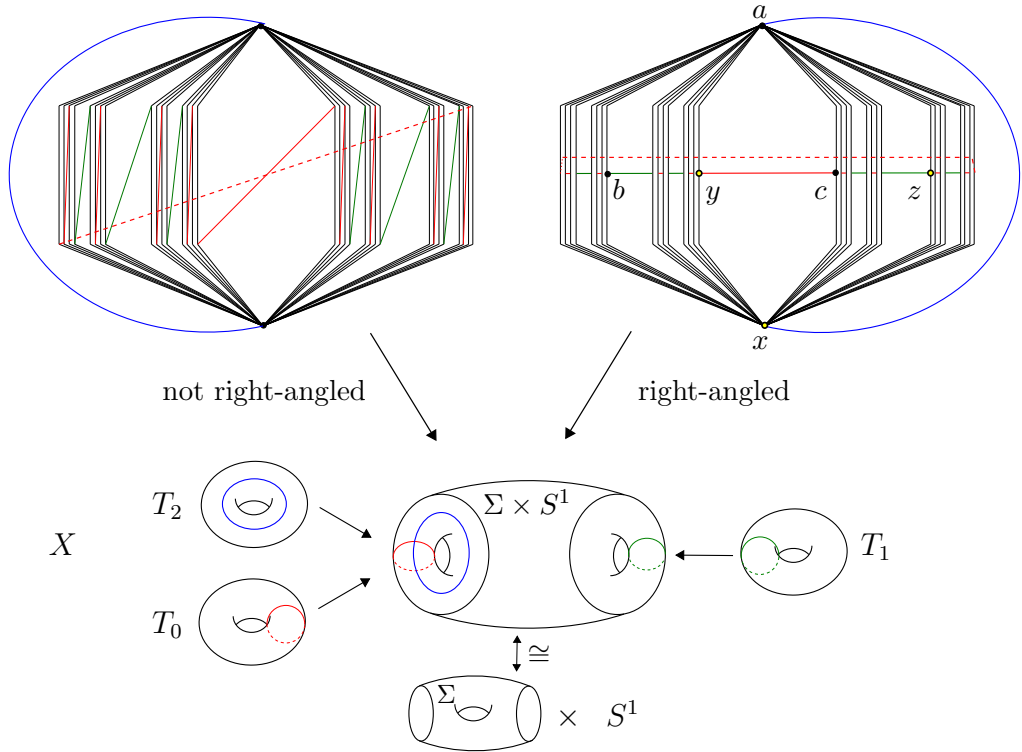

Let $\Sigma$ be a genus one surface with two boundary components. Let $\alpha_0$ and $\alpha_1$ denote the boundary curves of $\Sigma$. Form the product $\Sigma \times S^1$. Let $b_i \in \Sigma$ be a point on the boundary curve $\alpha_i$ for $i \in \{0,1\}$, and let $\beta_i = b_i \times S^1$. 

%%%%%% For current spacing purposes
%\clearpage

Form the space $X$ by gluing three tori, $T_0, T_1$, and $T_2$, onto the curves $\alpha_0, \alpha_1$, and $\beta_0$, respectively.  Note that after gluing on any two of these tori, the resulting space is homotopy equivalent to a $3$-manifold. 

\begin{Notation} Let $M = (\Sigma \times S^1) \cup_{\beta_0} T_2$, let $W_i = \alpha_i \times \beta_i$ for $i = 0,1$, and let $Y_i = W_i \cup_{\alpha_i} T_i$. 
\end{Notation}

\begin{lemma} We collect a few easy facts about  $\pi_1(X)$.
\begin{itemize}
\item The subgroups $\pi_1(M)$ and $\pi_1(Y_i)$ are isomorphic to the direct product of a free group with $\zz$. 
\item $\pi_1(X)$ has cohomological dimension $2$, and hence $\partial_\infty X$ is $1$-dimensional by \cite{bestvinamess}. 
\item $\pi_1(X)$ splits as an amalgamated product over each of the $\pi_1(\alpha_i)$.
\item The subgroup $\pi_1(Y_0 \cup_{\beta_0}T_2)$ is isomorphic to the Croke-Kleiner group.  
\end{itemize}
\end{lemma}

We define the \emph{angle} between $\alpha_0$ and $\beta_0$ to be the minimal distance in the Tits metric between the endpoints of their lifts in $\partial_\infty \widetilde W_0$.  

\begin{Theorem}\label{theorem_angles}
Given any $\delta \in (0, \frac{\pi}{2}]$, there is a locally $CAT(0)$ metric on $\Sigma \times S^1$ so that the angle between $\alpha_0$ and $\beta_0$ is $\delta$. 
\end{Theorem}

\begin{proof}
Choose generators $a,b,c$ for $\pi_1(\Sigma)$ so that $c$ represents $\alpha_0$, and let $d$ be the generator of $\pi_1(S^1)$. Assume that $\Sigma$ has a hyperbolic metric; so, there is an discrete and faithful representation $\psi: \pi_1(\Sigma) \rightarrow \Isom^+(\mathbb{H}^2)$.  

Let $T$ be a nontrivial homomorphism from $\mathbb{Z} \rightarrow \rr$, and skew the standard product action of $\pi_1(\Sigma) \times \zz$ on $\mathbb{H}^2 \times \mathbb{R}$ by:
\begin{align*}
  a & : (x,y) \mapsto (\psi(a)x, y) \\
  b& : (x,y) \mapsto (\psi(b)x, y) \\
  c& : (x,y) \mapsto (\psi(c)x, T(c)y) \\
 d& : (x,y) \mapsto (x,y+1).
\end{align*}

Since we can choose any homomorphism $T$, the quotient of $\mathbb{H}^2 \times \rr$ by this skewed action defines a locally CAT(0) metric with a desired angle.
\end{proof}

If the angle between $\beta$ and $\alpha_0$ is $\pi/2$ (so that $T = 0$), then we say that $X$ is \emph{right-angled}. We will always choose simple closed curves on the $T_i$ and flat metrics so that the $T_i$ are glued onto $\Sigma \times S^1$ by isometries. The resulting complex is then locally CAT(0). 

\begin{Remark}
We were originally motivated by an example of Kapovich and Kleiner~\cite[Section 9]{kapovichkleiner}. Roughly speaking, their example is obtained by gluing together two graph manifolds with boundary along simple closed curves in the interior. They use coarse Alexander duality to show that their examples are not virtually 3-manifold groups. Their examples also have visual boundaries which are nonplanar, and the existence of nonplanar graphs depends on how the Seifert-fibered pieces are glued together. 

Having two boundary components in our example is a little artificial; there are similar constructions formed by starting with the product of a circle and a genus one surface with one boundary component, and then gluing on two tori along the boundary torus. The difference here is that one cannot skew the metric in the same way; any locally CAT(0) metric on this space restricts to the right-angled flat metric on the boundary torus. Therefore, to obtain different boundaries one has to vary the fundamental groups of these examples by gluing tori onto different simple closed curves on the boundary tori. One can show using work of  Kapovich--Leeb~\cite[Proposition 2.2]{kapovichleeb} that all these groups are quasi-isometric.
We wanted an example where this phenomena occurs for the same group, hence used two boundary components. 
\end{Remark}

\subsection{Structure of the universal cover of $X$}

Let $\pi: \widetilde X \rightarrow X$ denote the universal covering. A \emph{block} is a connected component of $\pi^{-1}(M)$ or $\pi^{-1}(Y_i)$. Each block is convex in $\widetilde X$. A \emph{sub-block} is a connected component of $\pi^{-1}(\Sigma \times S^1)$. A \emph{wall} is a connected component of $\pi^{-1}(W_i)$. Note that each wall is contained in exactly two blocks, and two blocks either intersect in a wall or are disjoint. In the first case we say that the blocks are \emph{adjacent}. 

Each block and sub-block is quasi-isometric to the product of a tree with $\rr$, so the boundary of each block and sub-block is homeomorphic to the suspension of a Cantor set. There is a natural homeomorphism from the boundary of $\widetilde \Sigma$ to the Cantor set that maps the endpoints of lifts of the curves $\alpha_i$ to endpoints of removed intervals. A \emph{pole} of a block boundary is one of the suspension points. A \emph{longitude} is an embedded arc in the block boundary connecting the two poles. If the longitude of a sub-block covering $\Sigma \times S^1$ contains a pole of an adjacent block in its interior, we say it is a \emph{boundary longitude}, and if a pair of such longitudes contains the two poles of an adjacent block, we say the longitudes form a \emph{boundary pair}. The boundary pairs are precisely the pairs of longitudes which contain the endpoints of lifts of the curves $\alpha_i$.

We will also need the notion of an itinerary for a geodesic ray in $X$. Choose a basepoint $p \in \widetilde X$ not contained in any wall. For a point $\psi$ in $\partial_\infty \widetilde X$, the \emph{$p$-itinerary of $\psi$}, denoted $\Itin(\psi)$, is the ordered sequence of blocks $\{B_i\}_{i = 1}^\infty$, where the geodesic ray $p\psi$ intersects $B_i$ in a point not contained in a wall. A boundary point has finite itinerary (for any basepoint $p$) if and only if it is contained in the boundary of some block.

\section{Paths in the boundary} \label{sec:paths}

In this section, we study the local path components of points in $\partial_\infty \widetilde X$ and show that they depend on whether $X$ is right-angled. In the next sections, we will show that this changes what possible finite graphs embed into $\partial_\infty \widetilde X$. The main difference between the cases comes from studying paths between points on different longitudes of $\partial_\infty \SB \subset \partial_\infty \widetilde X$, where $\SB$ is a sub-block covering $\Sigma \times S^1$. If $X$ is right-angled, it turns out that there is a path between points on any two longitudes that misses the two poles of $\partial_\infty \SB$ (see Lemma~\ref{l:paths}). If not, these paths only exist between points on boundary pairs. See Figure~\ref{figure:main_example}.

\subsection{Horizontal paths in $\partial_\infty \widetilde X$}

%We begin by studying the local path components of $\partial_\infty M_{pq})$ around poles of $T_{pq}$. We say a path $\rho: [-t,t] \rightarrow X$ in a space $X$ is \emph{locally contained} in a subspace $A \subset X$ if $\rho(0) = a \in A$ and there exists a neighborhood $U$ of $a$ in $X$ such that $\rho \subset U$ is contained in $A$. Otherwise, we say a path $\rho$ \emph{locally leaves} $A$. 

\begin{lemma}\label{l:paths}
Suppose that $X$ is right-angled, and $\SB$ is a sub-block of $\widetilde X$. Then there is a circle inside $\partial_\infty \widetilde X$ containing $\partial_\infty \widetilde \Sigma \subset \partial_\infty \SB$ that avoids the poles $\{\psi_+, \psi_-\}$ of $\partial_\infty \SB$. In particular, if $x$ and $y$ are two points in $\partial_\infty \SB$, then $x$ and $y$ are in the same path component of $\partial_\infty \widetilde X - \{\psi_+, \psi_-\}$.
\end{lemma}

\begin{proof}
Choose a basepoint $p$ in the interior of $\widetilde \Sigma$, for some lift $\widetilde{\Sigma}$ of $\Sigma$ in $\SB$, and choose a circle in the tangent space $T_p(\widetilde \Sigma) \subset T_p(\SB)$. We will define an embedding from this circle to $\partial_\infty \widetilde X$. For a vector $v \in T_p(\widetilde \Sigma)$, if the horizontal geodesic $\rho$ in the direction of $v$ stays in $\SB$ for all time, then map $v$ to the endpoint of that geodesic. Otherwise, suppose that $\rho$ hits a wall $W$ between $\SB$ and a block $Y$ in a geodesic line~$\ell$, where $\pi(Y) = Y_i$ for $i \in \{0,1\}$. The line $\ell$ is a lift of a closed curve in the torus $T_i$, and hence $\ell$ bounds two half-flats meeting $W$ in $\ell$. Given $\ell$, choose one of the two half-flats; this choice determines a way to extend $\rho$ to a geodesic ray $\rho'$ contained in $Y$. Map the vector $v$ to the endpoint of~$\rho'$. 
\end{proof}

We call such a circle in $\partial_\infty \widetilde X$ a \emph{horizontal circle} for the sub-block boundary $\partial_\infty \SB$. Note that a path connecting two points on $\partial_\infty \SB$ which lies on this horizontal circle will have infinite Tits length unless $x$ and $y$ are on the same longitude, or are on longitudes that form a boundary pair.

\subsection{Local path components in $\partial_\infty \widetilde X$ if $X$ is not right-angled}

\begin{Theorem}\label{l:paths5}
Suppose $X$ is not right-angled. 
Suppose that $\psi$ is a point on a block boundary $\p B$ that is not a pole of any other block. Then, there is a neighborhood $\Omega$ of $\psi$ in $\partial_\infty \widetilde X$ so that the path component of $\psi$ in $\Omega$ is contained in the block boundary $\p B$. Furthermore, this path component is contained in the set of longitudes of $\p B$ that $\psi$ lies on.
\end{Theorem}

\begin{proof}

Our argument is the same as in \cite[Lemma 4]{crokekleiner}, though we have to additionally argue that the horizontal paths constructed in Lemma \ref{l:paths} cannot exist when $X$ is not right-angled. 

\noindent \textbf{Case $1$) $\psi$ is a pole of $\partial_\infty B$ where $B$ covers $M$.} Choose a basepoint $p$ in $B$ that is not contained in a wall, and suppose the minimum Tits angle based at $p$ between $\psi$ and the poles of adjacent blocks is equal to $\epsilon > 0$. Let $$\gO = \{\psi' \in \partial_\infty \widetilde X \,|\, \angle_p(\psi, \psi') < \epsilon/2\}.$$ For $i = 0$ or $1$, choose a lift $\widetilde \alpha_i$ of $\alpha_i$ in $B$, and let $\gO_{\widetilde \alpha_i}$ denote the geodesics in $\gO$ that exit from this lift. We claim that these sets are open and closed in~$\gO$. 
\vspace{1mm}

\noindent \emph{Open:} If $\psi'$ in $\gO_{\widetilde \alpha_i}$, note that any sufficiently close ray $\psi''$ to $\psi'$ must exit $B$ at a point close to where $\psi'$ exits. Since the lifts of $\alpha_i$ in $B$ are discrete, $\psi''$ and $\psi$ must leave from the same lift.
\vspace{1mm}

\noindent \emph{Closed:} Let $E$ denote the set of exit points in $\widetilde \alpha_i$ for elements of $\gO_{\widetilde \alpha_i}$. 
Then $E$ is bounded, for otherwise we could find a sequence of exit points on $\widetilde \alpha_i$ diverging from $p$, so we get a limit geodesic originating at $p$ and ending at $\psi'$ in $\partial_\infty \widetilde \alpha_i$. 
However, this point is a pole of an adjacent block, and hence the Tits angle at $p$ is $\ge \epsilon$, contradiction. 
Therefore, the set $E$ is bounded. 

Now, suppose we have a sequence $\psi_k' \rightarrow \psi'$, where $\psi_k' \in \gO_{\widetilde \alpha_i}$ and $\psi' \in \gO$. After passing to a subsequence, we can assume the geodesic segments $\bar{pe_k}$ converge to a segment $\bar{pe_\infty}$. Therefore, the geodesic $p\psi'$ exits $B$ through $\bar{pe_\infty}$, and is therefore in $\gO_{\widetilde \alpha_i}$. 

It follows that the path component of $\psi$ in $\gO$ consists of geodesics which never leave through lifts of the $\alpha_i$, since any subset $C \subset \gO$ containing $\psi$ and intersecting $\gO_{\widetilde \alpha_i}$ for some lift admits a separation into open subsets of $C$, and any $\psi' \in \gO - \partial_\infty B$ lies in $\gO_{\widetilde \alpha_i}$ for some lift.

\noindent \textbf{Case 2) $\psi$ is contained in the boundary of a wall and not a pole.}  A similar argument to the above shows that a path of geodesics starting at $\psi$ does not exit the sub-block containing $W$ through a lift of $\beta$ or $\alpha_i$. Therefore, the path component in $\gO$ is contained in the boundary of that sub-block. Since the boundary of this sub-block is homeomorphic to a suspension of a Cantor set, the path component in $\gO$ is contained in $\partial_\infty W$. 

\noindent \textbf{Case 3) $\psi$ is not in the boundary of an adjacent block.} Choose a basepoint $p$ in $B$ and not contained in a wall. Let $H \subset B$ be the halfplane in $\widetilde X$ containing the ray $p\psi$ and the vertical geodesic through $p$ between the poles of $B$.  For any wall $W$, the space $W \cap H$ is either empty, a vertical geodesic, or a flat strip bounded by vertical geodesics. After removing these subsets, we get either an infinite collection of open strips or a finite collection of strips and an open half plane. The second case occurs only for geodesics in the boundary of sub-blocks.
First suppose there is no half-plane. 
If $S$ is such a strip, then let $\gO_S$ be the points in $\partial_\infty X$ which intersect $S$. Again, this set is closed and open in $\gO$. Therefore, since $p\psi$ intersects all strips, each point in a path starting at $p$ must intersect them all as well. Therefore, the path is contained in $\partial_\infty H$. 

Now, suppose there is a half-plane.  Let $SB$ be the sub-block covering $\Sigma \times S^1$ which contains this half plane, and let $\rho$ be a path starting at $\psi$. For each geodesic ray $p\psi'$ for $\psi'$ in $\rho$, project its intersection with $\SB$ to the horizontal subsurface $\widetilde \Sigma$. If this projection is a single geodesic in $\widetilde \Sigma$ we are done. If not, then this path contains geodesic rays arbitrarily close to $p\psi$ whose endpoints are contained in the boundaries of walls and are not poles of adjacent blocks (note this fails if $X$ is right-angled). By Case $2$, the path components of these points in $\gO$ are contained on a single longitude, which is a contradiction.

\noindent \textbf{Case $4$) $\psi$ is a pole of $\partial_\infty Y$ where $Y$ covers $Y_i$}. Choose a basepoint $p$ in a lift $\widetilde \alpha_i$ of $\alpha_i$, so that $\psi$ is one of the endpoints of $\widetilde \alpha_i$. The argument in Case~$1$ above shows that the geodesics in any path starting at $\psi$ cannot exit $Y$ through a lift of $\beta$. The argument in the second half of Case $3$ rules out geodesics in this path exiting $Y$ into an adjacent sub-block through a wall. 
Therefore, the local path component around $\psi$ is contained in $\partial_\infty Y$. 
\end{proof}

We now use the analysis of local path components in the non right-angled case to work out the path components of $\partial_\infty \widetilde X - \{\psi_+, \psi_-\}$, where $\psi_+$ and $\psi_-$ are poles of a block boundary. 

\begin{lemma}\label{l:paths9}	
Suppose that $X$ is not right-angled and suppose that $Y$ is a block which covers $Y_0$ or $Y_1$. Let $\psi_+$ and $\psi_-$ be the poles of $\partial_\infty Y$, and suppose that $x$ and $y$ are points in $\partial_\infty Y - \{\psi_+, \psi_-\}$. Then $x$ and $y$ are in the same path component of $\partial_\infty \widetilde X - \{\psi_+, \psi_-\}$ if and only if they lie on the boundary of an adjacent sub-block $\SB$ or are contained in the same longitude. 
\end{lemma}

\begin{proof}

%First, suppose that $x$ and $y$ are contained on a block boundary of an adjacent sub-block. In this case there is a path from $x$ and $y$ to the poles of that block, and then a different longitude connects the two poles without touching $\{\psi_-, \psi_+\}$. 
Suppose that $x,y \in \p \widetilde{Y}$ are not contained in a single longitude or the boundary of an adjacent block. Suppose an arc $\rho$ starts at $x$ and ends at $y$. We must show that $\rho$ intersects $\psi_+$ or $\psi_-$. 
If $x$ or $y$ is not contained in the boundary of a wall, then the longitude containing $x$ or $y$ has no poles on it, which by Theorem~\ref{l:paths5} implies the path $\rho$ intersects $\psi_+$ or $\psi_-$. Therefore, we can assume that $x$ and $y$ are contained in two different sub-block boundaries $\SB_1$ and $\SB_2$. The group $G$ splits as $\pi_1(M) \ast_{\langle \alpha_i \rangle} T_i$, and the two sub-blocks $\SB_1$ and $\SB_2$ correspond to subgroups in different conjugates of $\pi_1(M)$. %\notakevin{Need to prove}
The $\SB_i$ are separated by a lift of $\alpha_i$ in $Y$, which implies that there is no path connecting $x$ and $y$ which misses $\psi_+$ or $\psi_-$. %\nota{Maybe say this as a lemma earlier? I'm sure it follows from papasoglu or swenson}
\end{proof}

\begin{lemma}\label{l:paths7}
Suppose that $X$ is not right-angled and suppose that $\{\psi_+, \psi_-\}$ are the poles of a block $B$ which covers $M$. Suppose $x$ and $y$ are two points in $\partial_\infty B - \{\psi_+,\psi_-\}$. Then $x$ and $y$ are in different path components of the space $\partial_\infty \widetilde X - \{\psi_+, \psi_-\}$ if and only if they do not lie on longitudes which form a boundary pair or are contained in the same longitude. 
\end{lemma}

\begin{proof}
%If $x$ and $y$ lie on longitudes which form a boundary pair, then there are arcs from $x$ and $y$ to the poles of an adjacent block covering $Y_0$ or $Y_1$. A different longitude of that block then connects the two poles and misses $\{\psi_+, \psi_-\}$. 

Suppose $L_1$ and $L_2$ are longitudes that do not form a boundary pair, and that there is an arc $\rho$ which starts at $x \in L_1$ and ends at $y \in L_2$.  We must show that $\rho$ intersects $\psi_+$ or $\psi_-$. This is certainly true if neither of the $L_i$ are boundary longitudes by Theorem~\ref{l:paths5}, as there are no poles of adjacent blocks on these longitudes.  So, we can assume that both of the points lie on boundary longitudes for different lifts of the $\alpha_i$.

If $\rho$ does not intersect $\psi_+$ or $\psi_-$, then $\rho$ must intersect a pole $\psi'_+$ of an adjacent block boundary $\partial_\infty Y$ covering $Y_i$ for $i \in \{0,1\}$. 
If $\rho$ continues in $\partial_\infty B$, it intersects $\psi_+$ or $\psi_-$, so we assume the path continues into $\partial_\infty Y$.  
If $\psi'_-$ is the other pole of $\partial_\infty Y$, then $\{\psi'_+,\psi'_-\}$ is a cut pair for $\partial_\infty\widetilde X$ as in Lemma~\ref{l:paths9}. Since $\rho$ exits $\partial_\infty B$, it enters a different path component of $\partial_\infty\widetilde X - \{\psi'_+,\psi'_-\}$ by Lemma~\ref{l:paths9}. Hence, $\rho$ intersects $\psi'_-$, since $x$ and $y$ are in the same component of $\partial_\infty\widetilde X - \{\psi'_+,\psi'_-\}$. Now, $\rho$ must continue in $\partial_\infty B$, since otherwise by repeating this argument $\rho$ would intersect $\psi'_+$ and not be an embedded path. However, then $\rho$ intersects one of the poles of $\partial_\infty B$ by Theorem~\ref{l:paths5}.
\end{proof}

\section{Non-planarity}\label{s:np}

\begin{Theorem} \label{theorem_nonplanar}
For any locally CAT(0) metric on $X$, the visual boundary $\partial_\infty \widetilde X$ is nonplanar. %\notakevin{I think this is right, but am not sure. I also want to say this more generally a QI invariant for G.}
\end{Theorem}

\begin{proof}
Suppose that $h: \partial_\infty\widetilde X \rightarrow S^2$ is an embedding. Let $\SB$ be a sub-block covering $\Sigma \times S^1$, and let $\{\psi_+,\psi_-\}$ be the poles of $\partial_\infty \SB$. Let $\mathcal{P}$ be the collection of poles of adjacent blocks to $\SB$. For each pair of poles in $\mathcal{P}$, the longitudes from this pair to $\psi_+$ and $\psi_-$ give an embedded circle inside $\partial_\infty \SB$. 

We first claim that the images of each of these circles under the embedding $h$ must bound a disc inside $S^2 - h(\partial_\infty SB)$ if $h$ extends to an embedding of $\partial_\infty \widetilde X$. To see this, let $\{\psi'_+, \psi'_-\}$ be a pair of poles in $\mathcal{P}$ such that the corresponding circle $C$ does not bound a disc in $S^2 - h(\partial_\infty \SB)$. This implies there are longitudes $\ell_1$ and $\ell_2$ in $\partial_\infty \SB$ which connect $\psi_+$ to $\psi_-$ and map to different components of $S^2 - C$. In $\partial_\infty \widetilde X$, we can connect $\psi'_+$ to $\psi'_-$ with a longitude of an adjacent block. The image of this longitude under $h$ must intersect one of the $h(\ell_i)$, which is a contradiction.

We now claim there is no embedding of $$\partial_\infty SB \cup \ell$$ into $S^2$ such that each embedded circle as above bounds a disc, where $\ell$ is a longitude of the block $B$ containing $\SB$ and not contained in $\partial_\infty \SB$. Note that under the usual identification of $\partial_\infty \widetilde \Sigma$ with the Cantor set, the pairs in $\mathcal{P}$ are identified with endpoints of deleted intervals. 
If there was such an embedding $h$, then by attaching discs to the domain along these circles we could extend $\bar h$ to an embedding $S^2 \cup \ell \rightarrow S^2$, which is a contradiction. 
%Furthermore,  $g(\mathcal{P})$ must be strictly contained in $\mathcal{P}$, since we have to embed the curve $\ell$ inside $S^2$, and this cannot intersect any longitudes of $\partial_\infty Y$. 
%However, such a homeomorphism $g: C \rightarrow C$ does not exist. Indeed, let $H$ denote the following decomposition of $C$: $$H = \{\mathcal{P}\} \cup \text{singletons}.$$
%Then $C/H$ is homeomorphic to $S^1$. Thus, if $g(H)$ is the decomposition $$g(H) = \{g(P)\} \cup \text{singletons},$$ then $C/g(H)$ is also homeomorphic to $S^1$. However $g(\mathcal{P})$ is a proper subset of $\mathcal{P}$, and therefore misses a pair of endpoints $\{a,b\}$. Then the sets $$\{x \in C/H | x \le a\} \text{ and } \{x \in C/H | x \ge b\}$$ are a clopen decomposition of $C/H$, and hence $C/H$ is not connected. 
\end{proof}

If $X$ is right-angled, we can do better and find a nonplanar graph inside of $\partial_\infty \widetilde X$. 
\begin{Theorem}\label{t:K33}
If $X$ is right-angled, then there is an embedded $K_{3,3}$ graph inside $\partial_\infty \widetilde X$.
\end{Theorem}

\begin{proof}
An illustration of the $K_{3,3}$ subgraph on vertex sets $\{a,b,c\}$ and $\{x,y,z\}$ appears in Figure~\ref{figure:main_example}. Choose any block covering $M$, and let $b,y,c,z$ be points that lie, in this order, on the horizontal circle containing a copy of $\p \widetilde{\Sigma}$ guaranteed by Lemma~\ref{l:paths}. Let $a$ and $x$ be the suspension points of $\p M$. Then, there are longitudes in $\p (\widetilde{\Sigma \times S^1})$ connecting $a$ to $y$ and $z$, and longitudes connecting $x$ to $b$ and $c$. Furthermore, the horizontal circle gives paths from $b$ to $y$ and $z$ and from $c$ to $y$ and $z$. Finally, in $\p \widetilde{M}$, there is a longitude connecting $a$ and $x$ that does not intersect the horizontal circle. Thus, $\p \widetilde{X}$ contains an embedded copy of $K_{3,3}$.
%Choose any block covering $M$, and choose $4$ points $A,B,C$ and $D$ in $\partial M$ that lie on the horizontal circle containing $\partial_\infty \widetilde \Sigma$. Let $E$ and $F$ be the suspension points of $\partial_\infty M$, so that we naturally have longitudes in $\partial_\infty \widetilde X$ connecting $A,B,C,$ and $D$ to $E$ and $F$.  This graph is of course planar. However, in $\partial_\infty \widetilde M$ we can connect $A$ and $B$ by a longitude which misses the horizontal circle. Now, this graph contains a copy of $K_{3,3}$, see Figure \ref{}. 
\end{proof}

\section{No nonplanar graphs if $X$ is not right-angled}\label{s:npg}

In this section, we use our analysis in Section~\ref{sec:paths} of the paths in $\partial_\infty \widetilde X$ to show the following theorem. 

\begin{Theorem}\label{t:planar}
If $X$ is not right-angled, then there is no embedded nonplanar graph inside $\partial_\infty \widetilde X$. 
\end{Theorem}

In order to deal with the points in $\partial_\infty \widetilde X$ that are not contained in a block boundary, we need the following lemma, which is a slightly refined version of \cite[Lemma 7]{crokekleiner}. 

\begin{lemma}\label{l:itineraries}
Choose a basepoint $p$ in $X$ as above and $\psi$ a point in $\partial_\infty X$ with infinite $p$-itinerary $(B_0, B_1, B_2, \dots)$. Let $\rho:[0,1] \rightarrow \p X$ be a path in $\partial_\infty X$ with $\rho(0) = \psi$ so that there exists points in $\rho$ with different $p$-itinerary than $\psi$. Then there is $N \in \mathbb{N}$ so that for each $n > N$, there is a point in $\rho$ with finite $p$-itinerary $(B_0, B_1, \dots B_n)$. In particular, these points in $\rho$ are contained in $\partial_\infty B_n$. 		
\end{lemma}

\begin{proof}
Since all of the points in the path $\rho$ do not have the same itinerary, there is a block $B_N\in \Itin(\psi)$ so that $B_N$ is not in the itinerary of $\rho(t)$ for all $t\in [0,1]$. Let $$t_0 = \inf\, \{\,t \,\,|\,\,B_N \notin \Itin(\rho(t))\,\}.$$ 
By assumption on $t_0$, the block $B_N$ is in the $p$-itinerary of $\rho(t)$ for all $t < t_0$. For $t < t_0$, let $x_t$ be the point on the geodesic $\overline{p\rho(t)}$ which lie on the wall between $B_{N-1}$ and $B_N$. The points $x_t$ lie on a lift of $\alpha_i$ for some $i \in \{0,1\}$, and since the lifts are discrete, the geodesics $\overline{p\rho(t)}$ must all leave from the same lift. If the set $\{x_t\}$ were contained in a bounded subset of this lift, then $p\rho(t_0)$ would intersect the interior of $B_N$, contradicting our assumption. Therefore, there is a subsequence of $\{x_t\}$ that converges to an endpoint of a lift of the $\alpha_i$, which implies that $\rho(t_0)$ is precisely the point in the boundary corresponding to this endpoint, and hence is a pole in $\partial_\infty B_N$. We can repeat this argument for any block $B_n$ for $n > N$, since all of these blocks are not contained in $\rho$ for all $t$. 
\end{proof}

\begin{proof}[Proof of Theorem \ref{t:planar}]
A graph $H$ is a \emph{topological minor} of a graph $G$ if $G$ contains a subdivision of $H$ as a subgraph. Kuratowski's theorem \cite{kuratowski} states that a finite graph is nonplanar if and only if it does not have $K_{3,3}$ or a $K_5$ as a minor, so we have to rule out these two graphs. 
We give the proof for $K_{3,3}$; the proof for $K_5$ is similar. 

First assume that a vertex of $K_{3,3}$ is mapped into a block boundary. In this case, the vertex is mapped to a pole of that block by Theorem~\ref{l:paths5}. 

Suppose that a vertex $v$ is mapped to a pole $\psi_+$ in the boundary of a block $Y$ which covers $Y_i$ for some $i \in \{0,1\}$. Let $e_1, e_2$, and $e_3$ denote the edges containing $v$ in $K_{3,3}$, and let $\psi_-$ denote the other pole of $\partial_\infty Y$. If each $e_i$ maps to different components of $\partial_\infty \widetilde X - \{\psi_+, \psi_-\}$, then any loop $\lambda_{ij}$ in $K_{3,3}$ which contains $e_i$ and $e_j$ passes through $\psi_-$. However, there are three loops $\lambda_{ij}$ in $K_{3,3}$ so that $\lambda_{ij}$ connects $e_i$ with $e_j$ and so that their total intersection is $v$, which is a contradiction. 

Therefore, we can assume $e_1$ is contained in a different component of $\partial_\infty \widetilde X - \{\psi_+, \psi_-\}$ than $e_2$ and $e_3$, which, by Lemma \ref{l:paths9}, means that $e_2$ and $e_3$ lie on the boundary $\partial_\infty B$ of an adjacent sub-block $B$. As above, the loops $\lambda_{12}$ and $\lambda_{13}$ have to intersect at $\psi_-$, which implies that $\psi_-$ lies on $e_1$. Note that the images of $e_2$ and $e_3$ contain the poles of $\partial_\infty B$. 

Assume without loss of generality that the other endpoint $w$ of $e_1$ maps to $\psi_-$. There are paths in $K_{3,3}$ connecting $w$ to the endpoints of $e_2$ and $e_3$ that are disjoint from $v$, which implies by Lemma \ref{l:paths9} that the other edges of $K_{3,3}$ containing $w$ other than $e_1$ must map into $\partial_\infty B$. The images of these edges must contain the poles of $\partial_\infty B$, but in $K_{3,3}$ these edges are disjoint from $e_2$ and $e_3$, which is a contradiction.

Now suppose a vertex $v$ is mapped to a pole $\psi$ in the boundary of a block $B$ which covers $M$. Again, let $e_1, e_2$, and $e_3$ denote the edges leaving $v$, and $\psi_-$ the other pole. In this case, no pair of edges can be mapped to a boundary pair. If this occurred, there would be two points on these edges that are mapped to the poles of an adjacent block. This pair $\{\psi'_+, \psi'_-\}$ is a cut pair for $\partial_\infty X - \{\psi,\psi_-\}$, and no pair of points on these edges is a cut pair for $K_{3,3}$. Therefore, the paths would have to continue in the same component, which implies by Lemma \ref{l:paths9} that they continue on $\partial_\infty B$. This implies that they eventually intersect at $\psi_-$, which is a contradiction. 

Therefore, these paths must map to longitudes that are not boundary pairs. By Lemma \ref{l:paths7} these edges lie in different path components of the space $\partial_\infty \widetilde X - \{\psi_+,\psi_-\}$, so any loop between them has to pass through $\psi_-$. This again is a contradiction. 
 
We now suppose that all vertices in $K_{3,3}$ map to points in $\partial_\infty \widetilde X$ that are not contained in a block boundary, and hence have infinite $p$-itinerary for some (any) choice of basepoint $p \in \widetilde X$. 

We first claim that given a point $x \in \partial_\infty \widetilde X$ with infinite $p$-itinerary, there cannot exist three disjoint paths $\gamma_i$ starting at $x$ such that each point in the $\gamma_i$ has the same $p$-itinerary as $x$. Suppose that there are, and let $Y$ denote their union. The union of the geodesics from $p$ to points in $Y$ yields a proper map from $\text{Cone}_\infty(Y):= \left( Y \times [0, \infty)\right)/ \left(Y \times 0 \right) \rightarrow \widetilde X$.
Furthermore, there exists $t > 0$ so that $$(v \times t) \cap (e \times [t,\infty))= \emptyset$$ for every pair of disjoint vertex and edge $v, e \subset Y$. 

Choose a block in the $p$-itinerary of these points and a lift $\widetilde \alpha_i$ of $\alpha_i$ so that all geodesics in this cone exit the block through $\widetilde \alpha_i$. We can choose $\widetilde \alpha_i$ far enough away from $p$ so that for each edge $e \subset T$, the sector $e \times [t,\infty)$ intersects $\widetilde \alpha_i$ in an interval containing the exit point for the geodesic $\overline{px}$.  At least one of the exit points of the geodesics $\overline{pv}$ must be contained in the interior of these intervals, which is a contradiction.

We now see the itineraries of the vertices of $K_{3,3}$ are restricted. Suppose $x$ is a vertex of $K_{3,3}$ with fixed $p$-itinerary. The previous paragraphs imply there are three edges in $K_{3,3}$ where one vertex of the edge has the same $p$-itinerary as $x$ and the edge does not consist entirely of points with the same itinerary $p$-itinerary as $x$. Lemma \ref{l:itineraries} implies there is a block $B$ and a vertex in the interior of each of these edges mapping to $\partial_\infty B$. The endpoints of these edges are not in $\partial_\infty B$, so each of these edges must enter and exit $\partial_\infty B$, intersecting a pole of $B$. Since there are three edges and only two poles, this is a contradiction. 
\end{proof}

\section{Finite graphs in Croke-Kleiner boundaries}\label{s:ckfinite}

We now prove Theorem \ref{theorem_main2} from the Introduction. We first briefly recall some of the details of the Croke-Kleiner construction. The $2$-complex $X_{CK}$ they consider is precisely $Y_0 \cup_\beta T^2$. Again, the homeomorphism type of $\partial_\infty \widetilde X_{CK}$ depends on the angle between $\alpha_0$ and $\beta$. 

The universal cover $\widetilde X_{CK}$ again decomposes into blocks meeting along walls in a similar way. (The blocks are preimages of $Y_0$ or $W_0 \cup_\beta T^2$ and the walls are preimages of $W_0$.)

In $\partial_\infty \widetilde X_{CK}$, the local path components of any pole are contained in the block boundary \cite[Lemma 4]{crokekleiner}. In particular, there are no paths with the same behavior as our ``horizontal circles". This implies the following lemma.

\begin{lemma} \label{lemma:inf_tits_len}
Suppose an arc $\rho: [0,1] \rightarrow \partial_\infty \widetilde X_{CK}$ has infinite Tits length and $\rho(0)$ is contained in the union of the block boundaries. Then $\rho$ contains a point not contained in the union of the block boundaries. 
\end{lemma}
\begin{proof}
Suppose $\rho \subset \cup_{B \in \mathcal{B}} \partial_\infty B$, where $\mathcal{B}$ is the set of all blocks in $\widetilde{X}_{CK}$. For any $t \in [0,1]$ there is a neighborhood $U$ of $\rho(t)$ and a block $B_0 \in \mathcal{B}$ such that $\rho \cap U \subset \partial_\infty B_0$ by \cite[Lemma 4]{crokekleiner}. These neighborhoods cover $\rho$, and taking a finite subcover guarantees that $\rho$ is contained in a finite union of block boundaries, and therefore has finite length in the Tits metric. 
\end{proof}

\begin{Theorem}
Suppose $X_1$ and $X_2$ are homeomorphic to the Croke--Kleiner complex $X_{CK}$ and equipped with locally $CAT(0)$ metrics. If $\gG$ is a finite graph contained in $\partial_\infty \widetilde{X}_1$, then there is an embedding of the graph $\gG$ into $\partial_\infty \widetilde{X}_2$. 
\end{Theorem}

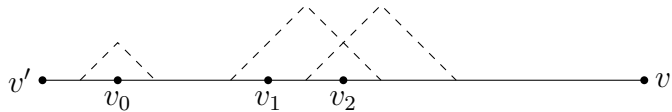
\begin{figure} \label{figure:combine_paths}
\begin{center}
\begin{tikzpicture}
			\draw (-4,0) -- (4,0);
			\draw[dashed] (-3.5, 0) -- (-3,.5) -- (-2.5, 0);
			\fill (-4,0)circle[radius=1.5pt];
			\fill (4,0)circle[radius=1.5pt];
			
			\fill (-3,0)circle[radius=1.5pt];
			
			\fill (-1,0)circle[radius=1.5pt];
			\fill (-0,0)circle[radius=1.5pt];
			
			\draw[dashed] (-1.5, 0) -- (-.5, 1) -- (.5, 0);			
			\node[right] at (4, 0) {$v$};
			\node[left] at (-4, 0) {$v'$};
			\node[below] at (-3, 0) {$v_0$};
				\node[below] at (-1, 0) {$v_1$};
					\node[below] at (-0, 0) {$v_2$};
					\draw[dashed] (-.5, 0) -- (.5, 1) -- (1.5, 0);

\end{tikzpicture}
\end{center}
\caption{Modifying the path between $v$ and $v'$ to have finite Tits length. The dashed lines represent the new paths constructed between poles of the same block. }
\end{figure}

\begin{proof}
A theorem of Xie~\cite[Theorem 6.1]{xie} implies that the {\it core} of the Tits boundaries of $\widetilde{X}_1$ and $\widetilde{X}_2$ are homeomorphic, where the core of the Tits boundary is the union of all topological circles in the boundary. Thus, if the graph $\Gamma \subset \p \wtX_1$ is contained in the union of the block boundaries, then the theorem follows. We prove that if a graph $\Gamma$ embeds in $\p \wtX_1$, then $\Gamma$ embeds in the union of the block boundaries in $\p \wtX_1$.

The coning argument in the proof of Theorem~\ref{t:planar} implies that there are no circles in the boundary so that each point on the circle has infinite $p$-itinerary for some basepoint $p$. Moreover, any vertex of the graph $\Gamma$ of valence greater than two must be contained in the union of the block boundaries. Thus, we may assume the graph $\Gamma \subset \p \wtX_1$ contains a point contained in the union of the block boundaries. 

We show that any edge of $\gG$ of infinite Tits length can be replaced with a nearby edge of finite Tits length. We may assume the edge contains a point contained in the union of the block boundaries. If the edge has infinite Tits length, then the edge must contain a point $v$ with infinite itinerary by Lemma~\ref{lemma:inf_tits_len}. By \cite[Prop 3.14]{qing}, there are points arbitrarily close to $v$ with finite itinerary. The proof of Lemma \ref{l:itineraries} goes through to show that there is a block $B$ of $\wtX_1$ so that each edge of $\Gamma$ containing $v$ contains points in $\partial_\infty B$. Thus, by the same reasoning as in the proof of Theorem \ref{t:planar}, $v$ is a vertex of degree $\le 2$ in $\gG$ or contained in an edge of $\gG$. If the vertex $v$ has degree $1$, shorten the edge containing $v$ to an edge ending at a pole of $\partial_\infty B$, where $B$ is a block as above.

If the vertex $v$ has degree $2$ or is contained in an edge of $\Gamma$, choose a neighborhood $U$ of $v$ in $\p \wtX_1$ so that $U \cap \gG$ only intersects a neighborhood of $v$ in $\gG$. Choose a closed neighborhood of $v$ in $\Gamma$ homeomorphic to an interval so that the endpoints of the neighborhood are contained in $U$ and map to poles of the same block. Replace this path with a longitude in this neighborhood which connects these poles. 

Replace paths in this way for each vertex $v$ with infinite itinerary in the edge~$e$. The collection of these points is closed, so these replacements can be made for a finite number of points and will cover the edge (or at least the part of the edge which contains points with infinite itinerary).  Each pair of endpoints of the replacement paths form a cut pair in $\partial_\infty \widetilde X_{CK}$. Thus, if two replacements overlap, then they contain the same pole and can be combined, see Figure~\ref{figure:combine_paths}. Therefore, these replacements can be glued together to form a path with finite Tits length. Since the path contained a point in the union of the block boundaries, the path is contained in the union of the block boundaries as desired. 
\end{proof}

\begin{remark}
Note that the above proof works for embedding graphs in the boundaries of our main example, as long as $X$ is not right-angled. Therefore, any two of these non-right-angled boundaries contain the same set of finite graphs. We expect that no two of these boundaries are homeomorphic. 
\end{remark}

\begin{remark}
Once one knows that the finite graphs inside $\partial_\infty \widetilde X_{CK}$ can be ``pushed" into the union of block boundaries, it is easy to describe precisely which finite graphs live inside $\partial_\infty \widetilde X_{CK}$. A concise description is the following: given a graph $\Gamma$, let $D(\Gamma)$ denote the graph obtained by ``doubling the vertices", i.e. replacing each vertex $v \in \Gamma$ with $v^+$ and $v^-$, and letting $[v^{+/-},w^{+/-}]$ be an edge in $D(\Gamma)$ if and only if $[vw]$ is an edge in $\Gamma$. For example, an edge in $\Gamma$ becomes a $4$-cycle in $D(\Gamma)$. If $\Gamma$ is a tree, then $D(\Gamma)$ is planar. The finite subgraphs of $\partial_\infty \widetilde X_{CK}$ are precisely the finite subgraphs of some subdivision of $D(\Gamma)$ for $\Gamma$ a finite tree. 

\end{remark}

\bibliographystyle{alpha}
\bibliography{Ref}     

\end{document}